\documentclass[11pt]{amsart}
\usepackage{amssymb}
\usepackage{amsfonts}
\usepackage{graphicx}
\usepackage{epsfig}
\usepackage{amsmath}

\theoremstyle{plain}
\newtheorem{theorem}{Theorem}[section]
\newtheorem{proposition}[theorem]{Proposition}
\newtheorem{lemma}[theorem]{Lemma}
\newtheorem{corollary}[theorem]{Corollary}
\theoremstyle{definition}

\newcommand{\R}{\mathbb{R}}
\newcommand{\C}{\mathbb{C}}

\def\new#1{\textsl{\textbf{{#1}}}}

\def\<{\langle}
\def\>{\rangle}
\def\ch#1{{{\bf H}^{#1}_{\C}}}

\def\P{\mathbb P}

\begin{document}
\title[Complex Kleinian groups]{A characterization of complex hyperbolic Kleinian groups in dimension $3$ with trace fields contained in $\mathbb R$}

\author{Joonhyung Kim}
\address{Joonhyung Kim \\
    Department of Mathematics, Konkuk University\\
           1 Hwayang-dong, Gwangjin-gu\\
           Seoul 143-701, Republic of Korea}
\email{calvary\char`\@snu.ac.kr}%

\author{Sungwoon Kim}
\address{Sungwoon Kim \\
    Center for Mathematical Challenges, KIAS\\
    Hoegiro 85, Dongdaemun-gu\\
    Seoul 130-722, Republic of Korea}
\email{sungwoon\char`\@kias.re.kr}
        \date{}
        \maketitle

\begin{abstract}
We show that $\Gamma < \textbf{SU}(3,1)$ is a non-elementary complex hyperbolic Kleinian group in which $tr(\gamma) \in \R$ for all $\gamma \in \Gamma$ if and only if $\Gamma$ is conjugate to a subgroup of $\textbf{SO}(3,1)$ or $\textbf{SU}(1,1)\times\textbf{SU}(2)$.
\end{abstract}

\footnotetext[3]{\sl{The first author was supported by NRF
grant 2013-014376.}}
\footnotetext[4]{\sl{The second author was partially supported by the Basic Science Research Program through the National Research Foundation of Korea (NRF) funded by the Ministry of Education, Science and Technology (NRF-2012R1A1A2040663)}}

\section{Introduction}
Let $\Gamma < \textbf{SU}(2,1)$ be a non-elementary complex hyperbolic Kleinian group.
The \emph{trace field} of $\Gamma$ is the field generated by the traces of all the elements of $\Gamma$ over the base field $\mathbb Q$. Maskit \cite[Theorem V.G.18]{Ma} characterized non-elementary hyperbolic Kleinian groups of $\mathbf{SL}(2,\mathbb C)$ whose trace fields are contained in $\mathbb R$.
The condition that the trace field of $\Gamma$ is contained in $\mathbb R$ is equivalent to that $tr(\gamma) \in \mathbb R$ for all $\gamma\in \Gamma$.
In \cite{FLW}, X. Fu, L. Li and X. Wang showed that if $tr(\gamma) \in \R$ for all $\gamma \in \Gamma$, then $\Gamma$ is Fuchsian. Here, a complex hyperbolic Kleinian group in dimension $2$ is called \emph{Fuchsian} if it keeps invariant a disc in Riemann sphere. It is very natural to generalize this result and there are two ways to generalize it, which are either $\Gamma$ is a subgroup of $\textbf{SU}(n,1)$, where $n\geq3$ or $\Gamma$ is a subgroup of $\textbf{Sp}(n,1)$. In latter case, J. Kim proved in the case of  $\textbf{Sp}(2,1)$ in \cite{Kim}.

In this paper, we consider the same problem in the case that $\Gamma$ is a subgroup of $\textbf{SU}(3,1)$. Our main theorem is the following.
\begin{theorem}
Let $\Gamma < \mathbf{SU}(3,1)$ be a non-elementary complex hyperbolic Kleinian group. Then $tr(\gamma) \in \R$ for all $\gamma \in \Gamma$ if and only if  $\Gamma$ is conjugate to a subgroup of $\mathbf{SO}(3,1)$ or $\mathbf{SU}(1,1)\times\mathbf{SU}(2)$.
\end{theorem}
The rest of this paper is organized as follows. In $\S$2, we give some necessary preliminaries on complex hyperbolic spaces and in $\S$3, we prove the main theorem.

\section{Preliminaries}
\subsection{Complex hyperbolic space}
Let $\C^{n,1}$ be a $(n+1)$-complex vector space with a
Hermitian form of signature $(n,1)$. An element of $\C^{n,1}$ is a
column vector $z=(z_1,\ldots,z_{n+1})^t$. Throughout this paper, we choose
the second Hermitian form on $\C^{n,1}$ given by the matrix $J$
$$
J=\left[\begin{matrix} 0 & 0 & 1 \\ 0 & I_{n-1} & 0 \\ 1 & 0 & 0
\end{matrix}\right].
$$
Thus $\<z,w\>=w^*Jz=\overline{w}^tJz=z_1\overline{w}_{n+1}+z_2\overline{w}_2+\cdots+z_n\overline{w}_n+z_{n+1}\overline{w}_1$,
where $z=(z_1,\ldots,z_{n+1})^t, \ w=(w_1,\ldots,w_{n+1})^t \in \C^{n,1}$.

Recall that the \emph{Heisenberg group} is $\mathfrak{N}=\C^{n-1}\times\R$ with the group law
$$
(z,u)(z',u')=(z+z',u+u'+2\mathrm{Im}\langle\langle z,\overline{z}' \rangle\rangle),
$$
where $\langle\langle \ , \ \rangle\rangle$ is the standard Hemitian
product on $\C^{n-1}$. One model of a complex hyperbolic space
$\ch{n}$, which matches the second Hermitian form is the \new{Siegel
domain $\mathfrak{S}$}, which is parametrized in horospherical
coordinates by $\mathfrak{N} \times \R_{+}$,
$$
\psi:(z,u,v)\mapsto \left[\begin{matrix}
-\langle\langle z,z \rangle\rangle-v+iu
\\ \sqrt{2}z \\ 1\end{matrix}\right] \hbox{ for }
(z,u,v)\in\overline{\mathfrak{S}}-\{\infty\} \hbox{   ; }
\psi:\infty\mapsto\left[\begin{matrix} 1
\\ 0 \\ \vdots \\ 0\end{matrix}\right],
$$
where $\infty$ is a distinguished point at infinity. The boundary of
$\mathfrak{S}$ is given by $\mathfrak{N} \cup \{\infty\}$.
Furthermore $\psi$ maps $\mathfrak{S}$ homeomorphically to the set
of points $w$ in $\P \C^{n,1}$ with $\langle w,w\rangle<0$ and maps
$\partial \mathfrak{S}$ homeomorphically to the set of points $w$
in $\P \C^{n,1}$ with $\langle w,w\rangle=0$. 

There is a metric on $\mathfrak{S}$ called the Bergman metric and the holomorphic isometry group of $\ch{n}$ with respect to this metric is $\textbf{PU}(n,1)$.
The elements of $\textbf{PU}(n,1)$ are classified by their fixed points. An element $A \in \textbf{PU}(n,1)$ is called \emph{loxodromic} if it fixes exactly two points
of $\partial \ch{n}$, \emph{parabolic} if it fixes exactly one point of $\partial \ch{n}$, and called \emph{elliptic} if it fixes at least one point of $\ch{n}$.

Now let's consider $\textbf{SU}(3,1)$. A general form of an element
$B \in \textbf{SU}(3,1)$ and its inverse are written as
$$
B=\left[\begin{matrix} a & b & c & d \\ e & f & g & h \\ l & m & n & p \\ q & r & s & t
\end{matrix}\right], \ B^{-1}=\left[\begin{matrix} \overline{t} & \overline{h} & \overline{p} & \overline{d} \\ \overline{r} & \overline{f} & \overline{m} & \overline{b} \\ \overline{s} & \overline{g} & \overline{n} & \overline{c} \\ \overline{q} & \overline{e} & \overline{l} & \overline{a}
\end{matrix}\right].
$$
Then, from $BB^{-1}=B^{-1}B=I$, we get the following identities.
$$\begin{array}{lll}
a\overline{t}+b\overline{r}+c\overline{s}+d\overline{q}=1,& a\overline{h}+b\overline{f}+c\overline{g}+d\overline{e}=0,& a\overline{p}+b\overline{m}+c\overline{n}+d\overline{l}=0, \\ a\overline{d}+|b|^2+|c|^2+d\overline{a}=0,&
e\overline{t}+f\overline{r}+g\overline{s}+h\overline{q}=0,& e\overline{h}+|f|^2+|g|^2+h\overline{e}=1, \\
e\overline{p}+f\overline{m}+g\overline{n}+h\overline{l}=0,& l\overline{t}+m\overline{r}+n\overline{s}+p\overline{q}=0,&
l\overline{p}+|m|^2+|n|^2+p\overline{l}=1,\\ q\overline{t}+|r|^2+|s|^2+t\overline{q}=0,&
\overline{t}a+\overline{h}e+\overline{p}l+\overline{d}q=1,& \overline{t}b+\overline{h}f+\overline{p}m+\overline{d}r=0,\\
\overline{t}c+\overline{h}g+\overline{p}n+\overline{d}s=0,& \overline{t}d+|h|^2+|p|^2+\overline{d}t=0,&
\overline{r}a+\overline{f}e+\overline{m}l+\overline{b}q=0,\\ \overline{r}b+|f|^2+|m|^2+\overline{b}r=1,&
\overline{r}c+\overline{f}g+\overline{m}n+\overline{b}s=0,& \overline{s}a+\overline{g}e+\overline{n}l+\overline{c}q=0, \\ \overline{s}c+|g|^2+|n|^2+\overline{c}s=1, & \overline{q}a+|e|^2+|l|^2+\overline{a}q=0. &
\end{array}
$$

The following lemmas are needed for us.
\begin{lemma}[Lemma 5.3 in \cite{GP}]\label{lem:2.1}
Let $B$ in $\mathbf{SU}(3,1)$ be such that the trace of $B$ is real. Then the characteristic polynomial of $B$ is self-dual.
\end{lemma}

\begin{lemma}[Proposition 2.2 in \cite{Kim}]\label{lem:2.2}
For two nonzero complex numbers $a$ and $b$, if $ab$ and $a\overline{b}$ are all real, then either $a$ and $b$ are real or $a$ and $b$ are purely imaginary.
\end{lemma}
Note that $0$ is both a purely real and purely imaginary number.

\subsection{Cartan angular invariant}
The Cartan angular invariant is a well-known invariant in complex hyperbolic geometry, and here we give the definition and some properties which will be used in the proof of the main theorem. For more details, see \cite{Go}.

The Cartan angular invariant $\mathbb{A}(x)$ of a triple $x=(x_1,x_2,x_3) \in (\partial\ch{n})^3$ is defined to be
$$
\mathbb{A}(x)=\arg(-\langle \tilde{x}_1,\tilde{x}_2 \rangle\langle \tilde{x}_2,\tilde{x}_3 \rangle\langle \tilde{x}_3,\tilde{x}_1 \rangle),
$$
where $\tilde{x}_1,\tilde{x}_2,\tilde{x}_3$ are lifts of $x_1,x_2,x_3$ respectively. Then $\mathbb{A}(x)$ is independent of the choice of the lifts and $-\pi/2 \leq \mathbb{A}(x) \leq \pi/2$. Furthermore, $\mathbb{A}(x)$ is invariant under permutations of the points $x_i$ up to sign. 
\begin{proposition}
A triple $x=(x_1,x_2,x_3) \in (\partial\ch{n})^3$ lies in the boundary of a complex line if and only if $\mathbb{A}(x)=\pm\pi/2$, and lies in the boundary of a Lagrangian plane if and only if $\mathbb{A}(x)=0$.
\end{proposition}

\section{Proof of the main Theorem}
The ``if" part is clear because any element of $\textbf{SO}(3,1)$ or $\textbf{SU}(1,1)\times\textbf{SU}(2)$ has real trace, so we will prove the ``only if" part.



It is well-known that a non-elementary Kleinian group contains
infinitely many loxodromic elements(See \cite{Ma} or \cite{FLW}).
Now let $A$ be a loxodromic element fixing $\mathbf 0$ and $\infty$ where $\mathbf 0$ and $\infty$ denote the points of $\partial\ch{3}$ represented by $(0,0,0,1)$ and $(1,0,0,0)$ respectively. In
terms of matrices, due to the Lemma \ref{lem:2.1}, we can write
$$A=\left[\begin{matrix} u & 0 & 0 & 0 \\ 0 & e^{i\theta} & 0 & 0 \\
0 & 0 & e^{-i\theta} & 0 \\ 0 & 0 & 0 & 1/u \end{matrix} \right],$$
where $u>1$. Up to conjugacy, we can assume that $A\in \Gamma$.

\begin{lemma}\label{lem:3.1}
If $B=\left[\begin{matrix} a & b & c & d \\ e & f &
g & h \\ l & m & n & p \\ q & r & s & t \end{matrix}\right]$ is an arbitrary element of $\Gamma$, then $a$, $t$, and $f+n$ are real.
\end{lemma}
\begin{proof}
Since the trace of every element in $\Gamma$ is real, $tr(B)$ and $tr(AB)+tr(A^{-1}B)$ are real.
\begin{eqnarray*}
tr(B)&=&a+t+f+n, \\
tr(AB)+tr(A^{-1}B)&=&\left(u+\frac{1}{u}\right)(a+t)+2\cos\theta(f+n).
\end{eqnarray*}
Solving for $(a+t)$ and $(f+n)$, since
$\displaystyle{u+\frac{1}{u} > 2\cos\theta}$, we get that $a+t$ and $f+n$
are real. Now consider
\begin{eqnarray*} \lefteqn{\left(u-\frac{1}{u}\right)tr(B)+tr(AB)-tr(A^{-1}B)} \\  & & \ \ \ \ \ \ \ \ \ = 2\left(u-\frac{1}{u}\right)a+\left(u-\frac{1}{u}\right)(f+n)+2i(f-n)\sin\theta. \end{eqnarray*}
Since $(f+n)$ is real,
$\displaystyle{\left(u-\frac{1}{u}\right)a+i(f-n)\sin\theta =: y_1 \in \R}$.

Similarly, by considering
\begin{eqnarray*} \lefteqn{\left(u^2-\frac{1}{u^2}\right)tr(B)+tr(A^2B)-tr(A^{-2}B)} \\  & & \ \ \ \ \ \ \ \ \ =2\left(u^2-\frac{1}{u^2}\right)a+\left(u^2-\frac{1}{u^2}\right)(f+n)+2i(f-n)\sin2\theta,\end{eqnarray*}
we have $\displaystyle{\left(u^2-\frac{1}{u^2}\right)a+2i(f-n)\sin\theta\cos\theta=:y_2
\in \R}$.
Hence,
$$\displaystyle{\left(u+\frac{1}{u}\right)y_1-y_2=i(f-n)\sin\theta\left(u+\frac{1}{u}-2\cos\theta\right)
\in \R}.$$ Since $\displaystyle{u+\frac{1}{u} > 2\cos\theta}$,
$i(f-n)\sin\theta=:y_3$ is real, so
$\displaystyle{\left(u-\frac{1}{u}\right)a}=y_1-y_3$ is real and so $a$ is
real. Since $a+t$ is real, $t$ is also real.
\end{proof}

\begin{lemma}\label{lem:3.2}
Consider the matrices $A, B_1, B_2$ in $\mathbf{SU}(3,1)$.
$$
A=\left[\begin{matrix} u & 0 & 0 & 0 \\ 0 & e^{i\theta} & 0 & 0 \\ 0 & 0 & e^{-i\theta} & 0 \\ 0 & 0 & 0 & 1/u \end{matrix} \right], B_1=\left[\begin{matrix} a_1 & b_1 & c_1 & d_1 \\ e_1 & f_1 & g_1 & h_1 \\ l_1 & m_1 & n_1 & p_1 \\ q_1 & r_1 & s_1 & t_1 \end{matrix}\right], B_2=\left[\begin{matrix} a_2 & b_2 & c_2 & d_2 \\ e_2 & f_2 & g_2 & h_2 \\ l_2 & m_2 & n_2 & p_2 \\ q_2 & r_2 & s_2 & t_2 \end{matrix}\right],
$$
where $u>1$. Suppose that $A, B_1$ and $B_2$ are in $\Gamma$. Then $b_1e_2+c_1l_2, d_1q_2, r_1h_2+s_1p_2, q_1d_2, e_1b_2+l_1c_2+h_1r_2+p_1s_2, f_1f_2+g_1m_2+m_1g_2+n_1n_2$ are all real.
\end{lemma}
\begin{proof}
We already know that $a_1, a_2, t_1, t_2, f_1+n_1, f_2+n_2$ are real by Lemma \ref{lem:3.1}. Since $(1,1)$ entry of $B_1B_2$ and $B_1AB_2+B_1A^{-1}B_2$ are real, $a_1a_2+b_1e_2+c_1l_2+d_1q_2$ and $\displaystyle{\left(u+\frac{1}{u}\right)a_1a_2+2\cos\theta(b_1e_2+c_1l_2)+\left(u+\frac{1}{u}\right)d_1q_2}$ are real, so $(b_1e_2+c_1l_2)+d_1q_2$ and $\displaystyle{2\cos\theta(b_1e_2+c_1l_2)+\left(u+\frac{1}{u}\right)d_1q_2}$ are real. Solving for $d_1q_2$ and $(b_1e_2+c_1l_2)$, we get $d_1q_2$ and $(b_1e_2+c_1l_2)$ are real.

In a similar way, considering $(4,4)$ entry of the same elements of $\Gamma$, we get that $q_1d_2$ and $(r_1h_2+s_1p_2)$ are real. Also, considering the sum of $(2,2)$ entry and $(3,3)$ entry of the same elements of $\Gamma$, we see that $e_1b_2+l_1c_2+h_1r_2+p_1s_2, f_1f_2+g_1m_2+m_1g_2+n_1n_2$ are all real.
\end{proof}

\begin{corollary}
Let $B_1$ and $B_2$ be arbitrary elements of $\Gamma$ as written in Lemma \ref{lem:3.2}.
\begin{itemize}
\item[(a)] Putting $B_1=B_2$ in the lemma we see that $b_1e_1+c_1l_1, d_1q_1, r_1h_1+s_1p_1$ and $f_1^2+n_1^2+2m_1g_1$ are all real.
\item[(b)] Putting $B_2=B_1^{-1}$ in the lemma we see that $b_1\overline{r}_1+c_1\overline{s}_1$ and $d_1\overline{q}_1$ are all real.
\item[(c)] Either $d_1$ and $q_1$ are both real or else they are purely imaginary.
\end{itemize}
\end{corollary}

Part (c) follows from (a), (b) and Lemma \ref{lem:2.2}. By this corollary, we know that for any $B\in \Gamma$, either $(1,4)$ entry and $(4,1)$ entry of $B$ are both real or else they are purely imaginary.

It is easy to check that $\mathbf 0$ and $\infty$ are the fixed points of $A$.
Since a non-elementary complex hyperbolic Kleinian group contains infinitely many loxodromic elements with pairwise distinct axes, there exists a loxodromic element $B_0$ of $\Gamma$ such that the axes of $A$ and $B_0$ are different.
Write $$B_0=\left[\begin{matrix} a_0 & b_0 & c_0 & d_0 \\ e_0 & f_0 & g_0 & h_0 \\ l_0 & m_0 & n_0 & p_0 \\ q_0 & r_0 & s_0 & t_0 \end{matrix}\right].$$
Then we claim that $d_0 q_0 \neq 0$. If $d_0 =0$, then we get $h_0=p_0=0$ from the identity $\overline{t}_0d_0+|h_0|^2+|p_0|^2+\overline{d}_0t_0=0$. This implies $B_0$ fixes $\mathbf 0$. Similarly if $q_0=0$, it can be easily seen that $B_0$ fixes $\infty$. In other words, if $d_0q_0=0$, then $B_0$ fixes either $\mathbf 0$ or $\infty$. This means that $A$ and $B_0$ share one but both fixed points. However the subgroup generated by such $A$ and $B_0$ is not discrete, which contradicts that $\Gamma$ is discrete. Therefore the claim holds.
Now we will consider the following two cases separately.

\smallskip

{\bf Case I}: $d_0$ and $q_0$ are purely imaginary.\\
From the identity $a_0\overline{d}_0+|b_0|^2+|c_0|^2+d_0\overline{a}_0=0$, we have $b_0=c_0=0$ because $a_0\overline{d}_0+d_0\overline{a}_0=0$. Similarly, from identities $\overline{q}_0a_0+|e_0|^2+|l_0|^2+\overline{a}_0q_0=0$, $q_0\overline{t}_0+|r_0|^2+|s_0|^2+t_0\overline{q}_0=0$, and $\overline{t}_0d_0+|h_0|^2+|p_0|^2+\overline{d}_0t_0=0$, we get $e_0=l_0=0$, $r_0=s_0=0$ and $h_0=p_0=0$, respectively.
Hence $$B_0=\left[\begin{matrix} a_0 & 0 & 0 & d_0 \\ 0 & f_0 & g_0 & 0 \\ 0 & m_0 & n_0 & 0 \\ q_0 & 0 & 0 & t_0 \end{matrix} \right],$$ where $a_0, t_0$ are real and $d_0, q_0$ are purely imaginary. Furthermore, since $\det B_0=(a_0t_0-d_0q_0)(f_0n_0-g_0m_0)=1$, we have $f_0n_0-g_0m_0=1$ because $1=a_0\overline{t}_0+b_0\overline{r}_0+c_0\overline{s}_0+d_0\overline{q}_0=a_0t_0-d_0q_0$.

From $\overline B_0^t J B_0=J$, we have
$$\left[\begin{matrix} \overline a_0 & \overline q_0 \\ \overline d_0 & \overline t_0 \end{matrix} \right] \left[\begin{matrix} 0 & 1 \\ 1 & 0 \end{matrix} \right] \left[\begin{matrix} a_0 & d_0 \\ q_0 & t_0 \end{matrix} \right] =\left[\begin{matrix} 0 & 1 \\ 1 & 0 \end{matrix} \right], \ \left[\begin{matrix} \overline f_0 & \overline m_0 \\ \overline g_0 & \overline n_0 \end{matrix} \right] \left[\begin{matrix}  f_0 &  g_0 \\  m_0 &  n_0 \end{matrix} \right]= \left[\begin{matrix} 1 & 0 \\ 0 & 1 \end{matrix} \right],$$
where $a_0t_0-d_0q_0=f_0n_0-g_0m_0=1$.
This implies that $\left[\begin{matrix} a_0 & d_0 \\ q_0 & t_0 \end{matrix} \right] \in \mathbf{SU}(1,1)$ and $\left[\begin{matrix}  f_0 &  g_0 \\  m_0 &  n_0 \end{matrix} \right] \in \mathbf{SU}(2)$. Hence $B_0$ is an element of $\textbf{SU}(1,1)\times\textbf{SU}(2)$.

Now let $B=\left[\begin{matrix} a & b & c & d \\ e & f & g & h \\
l & m & n & p \\ q & r & s & t \end{matrix}\right]$ be any
other element of $\Gamma$. Then $a$ and $t$ are real. By Lemma \ref{lem:3.2}, $dq_0$ is real and so $d$ is purely imaginary
because $q_0$ is a non-zero purely imaginary number. From identities
$a\overline{d}+|b|^2+|c|^2+d\overline{a}=0$ and
$\overline{t}d+|h|^2+|p|^2+\overline{d}t=0$, we get
$b=c=p=h=0$. Similarly, since $d_0q$ is real and $d_0$ is a non-zero purely imaginary number, we have that
$q$ is purely imaginary and using some identities, we get
$e=l=r=s=0$. Using the same arguments as above, we conclude that $B$ is of the form
$$B=\left[\begin{matrix} a & 0 & 0 & d \\ 0 & f & g & 0 \\ 0 &
m & n & 0 \\ q & 0 & 0 & t \end{matrix}
\right],$$ where $at-dq=fn-gm=1$. Thus we can conclude that $\Gamma$ is a subgroup of $\textbf{SU}(1,1)\times\textbf{SU}(2)$ defined by $$\mathbf{SU}(1,1)\times\mathbf{SU}(2) : = \left\{ \left[\begin{matrix} a & 0 & 0 & d \\ 0 & f & g & 0 \\ 0 &
m & n & 0 \\ q & 0 & 0 & t \end{matrix}\right] : \ \left[\begin{matrix} a & d \\ q & t \end{matrix} \right] \in \mathbf{SU}(1,1), \ \left[\begin{matrix}  f &  g \\  m &  n \end{matrix} \right] \in \mathbf{SU}(2) \right\}$$

\smallskip

\textbf{Case II}: $d_0$ and $q_0$ are real.\\
Let $B=\left[\begin{matrix} a & b & c & d \\ e & f & g & h \\
l & m & n & p \\ q & r & s & t \end{matrix}\right]$ be any
other element of $\Gamma$. Then, according to Lemma \ref{lem:3.1}, $a$ and $t$ are real. By Lemma \ref{lem:3.2}, $dq_0$ and $qd_0$ are real. Since $d_0$ and $q_0$ are non-zero real numbers, $d$ and $q$ are real. Hence we know that $(1,1), (1,4), (4,1)$ and $(4,4)$ entries of any element of $\Gamma$ are real.
Let $B_1$ and $B_2$ be elements of $\Gamma$ as written in Lemma \ref{lem:3.2}.
Considering the $(1,4)$ entry of $B_1^{-1}B_2$, we have that $\bar t_1d_2 + \bar h_1 h_2 + \bar p_1 p_2 + \bar d_1 t_2$ is real. Noting that
$B \left[ \begin{matrix} 0 & 0 & 0 & 1 \end{matrix} \right]^t = \left[ \begin{matrix} d & h & p & t \end{matrix} \right]^t$ and $$\langle B_2 \left[ \begin{matrix} 0 & 0 & 0 & 1 \end{matrix} \right]^t , B_1 \left[ \begin{matrix} 0 & 0 & 0 & 1 \end{matrix} \right]^t \rangle = \bar t_1d_2 + \bar h_1 h_2 + \bar p_1 p_2 + \bar d_1 t_2,$$
it follows that
$\langle B_2 \left[ \begin{matrix} 0 & 0 & 0 & 1 \end{matrix} \right]^t , B_1 \left[ \begin{matrix} 0 & 0 & 0 & 1 \end{matrix} \right]^t \rangle $ is real for all $B_1, B_2 \in \Gamma$.

Let $V$ be the $\mathbb R$-linear span of $\{B\left[ \begin{matrix} 0 & 0 & 0 & 1 \end{matrix} \right]^t  : B\in \Gamma \}$. Then it can be easily seen that $V$ is totally real. Furthermore every element of $\Gamma$ stabilizes $V$. Therefore $\Gamma$ leaves a totally real subspace of $\ch{3}$ invariant.
This means that $\Gamma$ is conjugate to a subgroup of $\mathbf S(\mathbf O(2,1)\times \mathbf O(1))$ or $\mathbf{SO}(3,1)$. Since $\mathbf S(\mathbf O(2,1)\times \mathbf O(1))$ is a subgroup of $\mathbf{SO}(3,1)$, we finally conclude that $\Gamma$ is conjugate to a subgroup of $\mathbf{SO}(3,1)$.



\end{document}